\documentclass[12pt]{amsart}
\usepackage{amsmath}
\usepackage{palatino}
\usepackage{amsfonts}
\usepackage{amsthm}
\usepackage{amssymb}
\usepackage{amscd}
\usepackage[all]{xy}
\usepackage{enumerate}
\usepackage{graphicx}
\usepackage{xcolor}



\theoremstyle{plain}

\newtheorem{thm}{Theorem}[section]

\newtheorem{thmx}{Theorem}

\newtheorem{prop}[thm]{Proposition}

\newtheorem{cor}[thm]{Corollary}
\newtheorem{corx}{Corollary}

\newtheorem{lem}[thm]{Lemma}

\theoremstyle{definition}
\newtheorem{defn}[thm]{Definition}

\newtheorem*{ackn}{Acknowledgment}

\newtheorem{rmk}[thm]{Remark}

\newcommand{\sO}{\mathcal{O}}

\newcommand{\mF}{\mathbb{F}}

\newcommand{\mN}{\mathbb{N}}
\newcommand{\mP}{\mathbb{P}}

\newcommand{\mZ}{\mathbb{Z}}

\newcommand{\Pic}{\mathrm{Pic}\,}

\numberwithin{equation}{section}

\title{On the slope inequalities for extremal curves }

\author{Valentina Beorchia and Michela Brundu}

\address{Dipartimento di Matematica e Geoscienze,
  Universit\`a di Trieste\\ 
via Valerio 12/b, 34127 Trieste, Italy.\\
\texttt{beorchia@units.it, brundu@units.it}}

\begin{document}

\begin{abstract} 
The present paper concerns the question of the violation of the $r$-th inequality for extremal curves in $\mP^r$, posed in [KM]. We show that the answer is negative in many cases (Theorem \ref{mainthm} and Corollary \ref{fasciabianca}). The result is obtained by a detailed analysis of the geometry of extremal curves and their canonical model. As a consequence, we show that particular curves on a Hirzebruch surface do not violate the slope inequalities in a certain range (Theorem \ref{verylast}).
\end{abstract}

\maketitle

\markboth{}{On the slope inequalities for extremal curves}

\tableofcontents

%%%%%%%%%%%%%%%%%%%%%%%
%%%%%%%%%%%%%%%%% SECTION
%%%%%%%%%%%%%%%%%%%%%%%

\section{Introduction}

Let $X$ be a smooth and connected projective curve of genus $g \ge 3$ defined over an algebraically closed field with characteristic zero. For each integer $r \ge 1$, H. Lange and P. E. Newstead in \cite[Section 4]{LN} introduced the notion of {\it $r$-gonality} $d_r(X)$ of $X$, which is the minimal integer $d$ such that there exists a linear series $g^r_d$ (hence there is a rational map $X \dasharrow \mP^r$ of degree $d$). In particular, for $r=1$ we obtain the classical {\it gonality} $\gamma(X)$ of the curve $X$. The sequence $\{ d_r(X)\}_{r \ge 1}$ is called the {\it gonality sequence of $X$}. 

For any curve and any $r \ge g$, the numbers $d_r$ are known by the Riemann-Roch Theorem. Hence there are only finitely many interesting numbers in a gonality sequence and in \cite{LN} it is evident that these numbers are deeply related to each other. In particular, in many cases they satisfy the {\em $r$-th slope inequality}, that is
\begin{equation}\label{satisfy slope}
\frac{d_r(X)}{ r} \ge \frac{d_{r+1}(X) }{ r+1}
\end{equation}
and this has been widely studied also in \cite{LM}.

Observe that if $X$ does not satisfy some slope inequality, then the corresponding Brill--Noether number is negative (see Remark \ref{ossBN}). Consequently,
a Brill--Noether general curve must satisfy all slope inequalities. 
The same occurs also for very special curves like hyperelliptic, trigonal and bielliptic curves 
(see \cite[Remark 4.5]{LN}).

The gonality sequence of a curve $X$ is related to Brill--Noether theory of vector bundles on $X$ 
(see \cite{LN}). 
Moreover, if a curve $X$ satisfies the $p$-th slope inequality for any $p < n$, then semistable vector bundles of rank $n$ on X satisfy Mercat's Conjecture, which governs the dimension of their spaces of global sections (cf. \cite[Conjecture 3.1 and Corollary 4.16]{LN}).

These results motivate the discussion of slope inequalities for specific classes of curves. 

Some sporadic examples of curves violating some slope inequality can be found by in \cite{Ba1}, \cite{Ba2}, \cite{LM}, \cite{KM}, \cite{Pan} and 
various families have been detected in \cite{LM}.
 Among such examples there are {\em extremal curves}, that is curves attaining the Castelnuovo bound for the genus. In \cite[Theorem 4.13]{LM} the authors prove that extremal curves of degree $\ge 3r - 1$ in $\mP^r$ do not satisfy all slope inequalities.

Moreover, in \cite[Theorem 4.4 and Corollary 4.5]{KM} Kato and Martens prove that an extremal curve in $\mP^3$
 of degree $d \ge 10$ satisfies
\begin{equation}
\label{viol3}
\frac{d_3}{ 3 }< \frac{d_{4} }{ {4}}.
\end{equation}

In the same paper (see the final Questions), the authors propose a new investigation in this direction; more precisely, they pose the following

\medskip
\noindent {\bf Question -} 
Is it true that extremal curves of degree $d \ge 3r+1$ in $\mP^r$ satisfy
\begin{equation}
\label{viol?}
\frac{d_r}{ r }< 
\frac {d_{r+1} }{ {r+1}}
\end{equation}
for $r \ge 4$?
\medskip

In the present paper we show that the answer to the above question is, in many cases, negative. The main results are the following:

\begin{thmx}
Let $X$ be a $\gamma$-gonal extremal curve of degree $d \ge 3r - 1$ in $\mP^r$ ($r \ge 3$) where $\gamma \ge 4$. Then $d_r=d$.
If, in addition, $X$ is not isomorphic to a plane curve,
then $X$ satisfies the following conditions:
\begin{itemize}
\item[$i)$]
 $d_{r+1}\le d+\gamma -1$;
\item[$ii)$]
under some technical hypotheses on the congruence class of $d$, the $r$-th slope inequality holds, i.e.
$$
\frac {d_r} r \ge \frac {d_{r+1} } {r+1}.
$$ 
\end{itemize}
\end{thmx}
 For more details on the assumptions, see Theorem \ref{mainthm}.
As a consequence, the following result holds (see Corollary \ref{fasciabianca}):
\begin{corx}
Let $X$ be a $\gamma$-gonal extremal curve in $\mP^r$, where $\gamma \ge 4$.
If the degree $d$ satisfies
\begin{equation}
r(\gamma-1) \le d \le \gamma(r-1)+1 
\end{equation}
then the $r$-th slope inequality holds. 
\end{corx}

The technique involved in the proofs relies on the fact that
extremal curves are either isomorphic to a smooth plane curve or they lie on a Hirzebruch surface (see Theorem \ref{acgh1}). Then, in order to bound the requested gonality number, we consider the residual divisor to a point in a general hyperplane section. It turns out that the dimension of the linear system associated with such a residual divisor can be estimated by looking at the canonical model of the curve.

The organization of the paper is as follows.

Section 2 contains some preliminaries about extremal curves and curves on Hirzebruch surfaces. 

Section 3 summarizes results on the gonal sequence and the slope inequality which can be found in the literature. 

Section 4 investigates the geometry and gonality of extremal curves. They are particular curves on a Hirzebruch surface, but not so rare. Indeed, thanks to Theorem \ref{everex}, any smooth irreducible curve on a Hirzebruch surface, under a suitable assumption on its class, can be embedded in a projective space as an extremal curve.
Next we state and prove the main results, that is Theorem \ref{mainthm} and Corollary \ref{fasciabianca}.

In Section 5, we recall a result of \cite{LM}, where the authors show that in degree $d=3r-1$ the $r$-th slope inequality is violated by extremal curves and we prove, by an {\em ad hoc} argument, that the same holds in degree $d=3r-2$ 
if $r \ge 5$.

Finally, in Section 6, we focus on fourgonal extremal curves, and we show that certain foursecant curves on a Hirzebruch surface admit several embeddings as an extremal curve in some projective space. This allows us to exhibit specific classes of 
extremal curves whose $r$-th slope inequality holds for $r$ in a suitable interval (see Theorem \ref{verylast}).

\ackn

The authors are grateful to the anonymous Referees for the accurate reading of the first version of the present paper, and for the useful comments and suggestions.

This research is supported by funds of the
Universit\`a degli Studi di Trieste - Finanziamento di Ateneo per
progetti di ricerca scientifica - FRA 2022, Project DMG. The first author 
is supported by italian MIUR funds, PRIN project Moduli Theory and Birational Classification (2017), and
 is member of GNSAGA of INdAM.

%%%%%%%%%%%%%%%%%%%%%%%
%%%%%%%%%%%%%%%%% SECTION
%%%%%%%%%%%%%%%%%%%%%%%

\section{Notation and preliminary notions}

If $x$ is a positive real number, by $[x]$ we denote the {\em integer part} of $x$, i.e. the largest integer not exceeding $x$.

In this paper $\mP^{n}$ denotes the projective space over an algebraically closed field of characteristic zero.
We shall also use the following notation:

\begin{itemize}
\item {} given a projective scheme $Z\subseteq \mP^{n}$,
$\langle Z \rangle$ will denote the linear span of $Z$ in $\mP^{n}$;

\item {} by a {\em curve} $X$ we shall mean a smooth irreducible curve,
unless otherwise specified;

\item {} given a linear system $|D|$ on a curve $X$, we will denote by
$\varphi _{D}$ the morphism associated with $|D|$ and by $X_{D}=
\varphi _{D} (X)$ the image of $X$ under $\varphi _{D}$; in particular, if $X$ is a non--hyperelliptic curve of genus $g$,
$X_K \subset \mP^{g-1}$ will denote the {\em canonical model} of $X$;

\item {} we say that a linear series $g^r_d$ is {\em simple} if for any divisor $D \in g^r_d$
 the induced rational map $\varphi _{D}$ is birational onto its image.
\end{itemize}

\begin{defn}
\label{indeko}
 If $X$ and $Y$ are two curves, a morphism $\varphi: \; X \rightarrow Y$ is said {\em indecomposable} if it cannot be factorized as $\varphi = \alpha \circ \beta$, where $\alpha$ and $ \beta$ are morphisms of degree bigger than one. In particular, if $Y = \mP^1$, we say that a linear series $|D| = g^1_d$ is {\em indecomposable} if the morphism $\varphi _{D}$ associated with $|D|$ is so. 
\end{defn}

\begin{defn}
 The {\em gonality} of a curve $X$ is the minimum degree $d$ of a linear series $g^1_d$ on $X$; if $X$ has gonality $\gamma$ then each series $g^1_\gamma$ is called a {\em gonal series}. If $\Gamma \in g^1_\gamma$ denotes a general {\em gonal divisor} of $X$, then the morphism 
 $\varphi_\Gamma: X \rightarrow \mP ^1$ is called a {\em gonal cover}.
\end{defn}

 \medskip
Let $X$ be a curve in $\mP ^r$ of degree $d \ge 2r+1$ (the motivation of such a bound will be explained in the forthcoming Remark \ref{remarcoa}).
Setting 
$$
m=m_{d,r}=\left[ \frac {d-1} {r-1}
\right],
$$ 
 we can write, for a suitable integer $\epsilon=\epsilon_{d,r}$,
\begin{equation}
\label{degm}
d-1= m(r-1)+\epsilon , \quad 0\le \epsilon \le r-2.
\end{equation}
It is well--known that the genus $g(X)$ of $X$ satisfies the {\em Castelnuovo's bound} i.e.
\begin{equation}
\label{casbo}
g(X) \le \pi(d,r):=m\left( \frac{m-1} 2 (r-1)+\epsilon\right).
\end{equation}

Clearly, the values of $m$ and $\epsilon$ depend on $d$ and $r$. So we introduce the following notation.

\begin{defn}
\label{razio}
Let $X \subset \mP^r$ a curve of degree $d$. The integer $m=m_{d,r}$ in formula (\ref {degm}) will be called {\em $m$-ratio} of $X$ in $\mP^r$. Analogously, the integer $\epsilon=\epsilon_{d,r}$ 
will be called {\em $\epsilon$-remainder} of $X$ in $\mP^r$.
\end{defn}

\begin{defn}
\label{estre}
A curve $X$ is said an 
{\em extremal curve in $\mP ^r$} if it has a simple linear series $g^r_d$ of degree $d \ge 2r+1$ and $X$ has the maximal genus among all curves admitting such a linear series, i.e.
\begin{equation}
\label{maxgen}
g(X) = \pi(d,r)= m_{d,r}\left( \frac{m_{d,r}-1} 2 (r-1)+\epsilon_{d,r} \right).
\end{equation}
\end{defn}

Observe that the notion above is {\em relative} to the space $\mP^r$ where the curve lies.

\medskip
Finally, let us recall a few notions about rational ruled surfaces.

 \begin{defn} 
 We denote by $\mF_n:= \mP (\sO _{\mP^1} \oplus \sO_{\mP^1}(-n))$ a {\em Hirzebruch surface} of invariant $n$, by $C_0$ the unique (if $n >0$) unisecant curve with $C_0^2 <0$ and by $L$ a line of its ruling.
 \end{defn}

 It is well-known that $\Pic(\mF_n) = \mZ [C_0] \oplus \mZ [L ]$, where $C_0^2 = -n$, $C_0 \cdot L= 1$ and $L^2 =0$.

 \begin{rmk}
 \label{fritto}
 If $H = C_0 + \beta L$ is a very ample divisor on $\mF_n$, then the associated morphism $\varphi_H$ embeds $\mF_n$ in $\mP^r$ as a rational normal ruled surface $R$ of degree $r-1$, where $r:= H^2 +1$.
 With an easy computation, one can  see that $\beta = (r+n-1)/2$.

Finally, we recall that the canonical divisor of $R$ is 
 $K \sim -2H+(r-3)L$.
 \end{rmk}

In the sequel we will treat curves on (possibly embedded) Hirzebruch surfaces.

 \begin{defn} 
 \label{calcol}
 If $X \subset \mF_n$ is a $\gamma$-gonal curve, we say that {\em the gonality of $X$ is computed by the ruling of $\mF_n$} if $\gamma = X \cdot L$ (as far as $n \ge 1$).
\end{defn}

%%%%%%%%%%%%%%%%%%%%%%%
%%%%%%%%%%%%%%%%% SECTION
%%%%%%%%%%%%%%%%%%%%%%%

 \section {The slope inequalities}

 Let us recall a definition that generalizes the notion of the gonality of a curve.

\begin{defn} 
The {\it $r$-th gonality} of a curve $X$ is defined to be 
$$
\begin{aligned}
d_r:= & \quad \min\{d \in \mN \;| \; \hbox{$X$ admits a linear series $g^r_d$} \} = \\
{} = & \quad \min\{\deg L \;| \; \hbox{$L$ line bundle on $X$ with $h^0(L) \ge r+1$} \} .
\end{aligned}
$$
Moreover,  $(d_1, d_2, d_3, \dots, d_{g-1})$ is called the {\em gonality sequence} of $X$.
\end{defn}

Recall that
 the gonality sequence is strictly increasing and weakly additive (see \cite[Lemma 3.1]{LM}).

\begin{lem} 
\label{incre}
For any $\gamma$--gonal curve $X$ of genus $g$ and gonality sequence $(d_1, d_2, \dots, d_{g-1})$, the following properties hold:
\begin{itemize}
\item $d_1 < d_2 < d_3 < \cdots < d_{g-1}$, where $d_1 = \gamma$ and $d_{g-1}= 2g-2$;
\item $d_{r+s} \le d_{r} + d_{s}$ for all $r,s$; 
\item $d_r \le r \gamma$, for all $r$. 
\end{itemize}
\end{lem}

\begin{rmk}
The gonality sequence is defined up to $d_{g-1}$, since $d_r \ge 2g-1$ for $r \ge g$. Therefore, by the Riemann-Roch Theorem, $d_r = r+g$, for all $r\ge g$, 
as observed in \cite[Remark 4.4 (b)]{LN}.
\label{gimenuno}
 \end{rmk}

It is clear that, if the bound $d_r \le r \gamma$ is reached for all $r$ up to a certain integer $r_0$, then the sequence 
$$
\left( \frac{d_r} r \right)_{r=1, \dots, r_0}
$$ is constant and equal to the gonality $\gamma$. Otherwise, we can compare the above ratios: let us recall the following notion.

\begin{defn} 
 The relation 
\begin{equation}
\label{nonviol}
\frac{d_r}{ r }\ge \frac{d_{r+1} } {r+1}.
\end{equation}
is called the {\it ($r$-th) slope inequality}.
\end{defn}

\begin{rmk}
\label{ossBN}
Assume that, for some $r$, the slope inequality does not hold, i.e.
$$
\frac{d_r}{ r }< \frac{d_{r+1} } {r+1}
$$
hence $(r+1) d_r < r d_{r+1}$. Using this inequality in the computation of the corresponding Brill--Noether number (see, for instance, \cite{ACGH}), we obtain that
$$
\rho(d_r,r,g) = g-(r+1)(g-d_r+r) < r(d_{r+1} -g-r-1) \le 0.
$$
Therefore $\rho(d_r,r,g)$ is strictly negative.
\end{rmk}

\begin{defn} 
A curve $X$ of genus $g$ is called {\em Brill--Noether general } if $\rho(d_r,r,g) \ge 0$, for all $1 \le r \le g-1$.
\end{defn}

As a straightforward consequence, we have the following fact.

\begin{prop}
Let $X$ be a Brill--Noether general curve of genus $g$ and gonality $\gamma$. Then
\begin{equation}
\label{catena}
\gamma=d_1 \ge \frac{d_2 }{ 2} \ge \frac{d_3 }{ 3} \ge \frac{d_4 }{ 4} \ge \cdots \ge \frac{d_r }{ r} \ge \frac{d_{r+1} }{ {r+1} } \cdots
 \ge\frac{ d_{g-1} } {g-1} = 2,
\end{equation}
i.e. all the slope inequalities hold.
\end{prop}

\begin{rmk}
\label{rem2}
Nevertheless, also ``special" types of curves satisfy all the slope inequalities. For instance, in \cite[Remark 4.5]{LN}, one can find the explicit values of the gonality sequence of a $\gamma$-gonal curve $X$ in the following cases:

- if $\gamma =2$ ($X$ hyperelliptic);

- if $\gamma =3$ ($X$ trigonal);

- if $\gamma =4$ and $X$ is bielliptic;

- if $X$ is the general fourgonal curve.

\medskip
\noindent
In all the cases above, all the slope inequalities hold.

\end{rmk}

 For this reason, from now on we will assume
 $ \gamma \ge 4$.

 %%%%%%%%%%%%%%%%%%%%%%%
%%%%%%%%%%%%%%%%% SECTION
%%%%%%%%%%%%%%%%%%%%%%%

 \section {Extremal curves and gonality}

 Let us first recall a result (see \cite[III, Theorem 2.5]{ACGH}) that turns out to be important in the sequel since it describes the geometry of extremal curves. Let us keep the notation introduced in Section 2.

\begin{thm}
\label{acgh1} 
Let $d$ and $r$ be integers such that $r \ge 3$, $d \ge 2r+1$. Then extremal curves $X \subset \mP^r$ of degree $d$ exist and any such a curve is one of the following:
\begin{itemize}
\item[(i)] The image of a smooth plane curve of degree $k$ under the Veronese map 
$ \mP^2 \rightarrow \mP^5$. In this case $r=5$, $d=2k$.
\item[(ii)] A non-singular member of the linear system $|mH+L|$ on a rational normal ruled surface. In this case $\epsilon =0$.
\item[(iii)] A non-singular member of the linear system $|(m+1)H- (r-\epsilon-2)L|$ on a rational normal ruled surface. 
\end{itemize}
\end{thm}

In particular, any irreducible extremal curve is smooth.

\begin{rmk}
\label{remarcoa}
Observe that we assumed from the beginning that $d\ge 2r+1$. Namely, if $d< 2r$,
by Clifford Theorem 
the $g_d^r$ is non-special and we obtain $g=d-r$. 
In particular, we have $r>g$; by Remark \ref{gimenuno} the gonality sequence is known and the $r$-th slope inequality holds.

Moreover, if $d=2r$ then $m=2$ and $\epsilon =1$. Therefore $\pi(d,r) = r+1$. Hence, if $X$ is an extremal curve, then $d_r= d_{g-1}=2g-2=2r$. By Remark \ref{gimenuno}, we have $d_{r+1}= d_g= 2g=2r+2$, hence the $r$-th slope inequality holds.
\end{rmk}

\begin{rmk}
\label{remarcob}
In the sequel, we shall not consider the case (i) in the Theorem \ref{acgh1} where the extremal curve is the image of smooth plane curves under the Veronese map, being the gonal sequence of smooth plane curves completely understood by Max Noether's Theorem (see, for instance,
\cite[Theorem 3.14]{Cili}). 

More precisely, a plane curve of degree $k \ge 5$ satisfies
$$
d_r = \left\{
\begin{array}{ll}
\alpha d - \beta, & {\rm if}\ r < g=\frac{(d-1)(d-2)}{2}\\
r+g, & {\rm if}\ r\ge g,\\
\end{array}
\right.
$$
where $\alpha$ and $\beta$ are the uniquely determined integers with $\alpha \ge 1$ and $0 \le \beta \le \alpha$ such that $r=\frac{\alpha (\alpha +3)} {2}$.

In particular, as observed in \cite[Proposition 4.3]{LM}, whenever $\beta \neq 0$, the $r$-th slope inequality is satisfied, while if $\beta =0$ and $\alpha \le k-4$, such an inequality is violated.

In the case (i) of Theorem \ref{acgh1}, we have $r=5$, so $\alpha=2$ and $\beta =0$. It follows that
if $k \ge 6$, the $5$-th slope inequality is violated.
\end{rmk}

The converse of the cases $(ii)$ and $(iii)$ of Theorem \ref{acgh1} holds: namely the above classes of curves on a ruled surface force the curve to be extremal, as the following result shows.

\begin{prop}
\label{fiore}
Let $X \subset R \subset \mP^r$, where $X$ is a smooth curve of degree $d$ and $R$ is a rational normal ruled surface. 
Setting $m= m_{d,r}$ and $\epsilon= \epsilon_{d,r}$, if
$$
X \in 
\left\{
\begin{matrix} 
|mH+L| \hfill \\
or \hfill \\
|(m+1)H- (r-\epsilon-2)L|\\ 
\end{matrix}
\right.
$$
then $X$ is an extremal curve in $ \mP^r$.
\end{prop}
\begin{proof}
The canonical divisor $K$ of the surface $R$ can be written as $K \sim -2H + (r-3)L$. Therefore we can apply the Adjuncton Formula on $R$ (where $g$ denotes the genus of $X$):
$$
2g-2 = (K+X) \cdot X.
$$
In the first case $X \sim mH+L$ we then obtain
$$
2g-2 = ((m-2)H + (r-2) L) \cdot (mH+L).
$$
Taking into account that $H^2 =r-1$, we finally obtain
$$
2g = m(m-1)(r-1).
$$
On the other hand, $d= \deg(X) = X \cdot H = mH^2+L\cdot H = m(r-1)+1$, hence $d-1=m(r-1)$. Therefore, from \eqref{degm}, we have that $\epsilon =0$ and, so, $\pi(d,r)=g$ as requested (see (\ref{casbo})).

In the second case $X \sim (m+1)H- (r-\epsilon-2)L$ so we get
$$
2g-2 = ((m-1)H + (\epsilon -1) L) \cdot ((m+1)H- (r-\epsilon-2)L).
$$
It is immediate to see that
$$
2g = m(m-1)(r-1) +2m \epsilon,
$$
so again $g= \pi(d,r)$.
\end{proof}

The two results above characterize embedded extremal curves in terms of rational ruled surfaces.
Now we show that any smooth irreducible curve on a Hirzebruch surface, under a certain assumption on its class, can be embedded in a projective space as an extremal curve.
 
%%%%%%% ZONTABELLA
To do this,  the following known result  will be useful (see, for instance, \cite[Ch.V, Corollary 2.18]{H}).
\begin{prop}
\label{cuori}
Let $D$ be the divisor $aC_0 + b L$ on the rational ruled surface $\mF_n$. Then:
\item{(a)} $D$ is very ample $ \iff D$ is ample $\iff a>0$ and $b>an$;
\item{(b)} the linear system $|D|$ contains an irreducible smooth curve $\iff$ it contains an irreducible curve 
$\iff a=0, b=1$ or $a=1, b=0$ or $a>0, b>an$ or $a>0, b=an, n>0$.
\end{prop}
%%%%%%%

The Hirzebruch surface $\mF_0$ is isomorphic to $\mP^1 \times \mP^1$, so $\Pic (\mF_0)$ is generated by two lines belonging to distinct rulings. But also in this case, we denote these generators by $ C_0 $ and $L$ (like in the case of $\mF_n$, where $n >0$), even if their roles can be exchanged.

\begin{thm}
\label{everex}
Let $X \sim \gamma C_0 + \lambda L$ be an irreducible smooth curve on $\mF_n$ with $\gamma \ge 2$ and not isomorphic to a plane curve.

%\noindent
%Without loss of generality, we assume, if $n=0$, that 
%$\lambda \ge \gamma -1$ and, if $n=1$, that $\lambda > \gamma$.

\noindent
Dividing $\lambda -n-1$ by $\gamma -2$, let us denote by $\beta$ the quotient and $\epsilon$ the remainder, i.e.
$$
\beta:= \left[ \frac{\lambda -n-1}{\gamma -2}\right] =\frac{\lambda -n-1-\epsilon}{\gamma -2}, \quad 0 \le \epsilon \le \gamma -3.
$$
Moreover, set
$$
 r:=2\beta +1-n \quad \hbox{and} \quad d:=\gamma(\beta - n) + \lambda.
$$
Consider the complete linear system $|H|$ on $\mF_n$ given by
$$
H \sim C_0 + \beta L.
$$
Then
\begin{itemize}

\item[$i)$] $\beta > n$, for all $n \ge 0$;

\item[$ii)$] the morphism $\varphi_H$ embeds $\mF_n$ in $\mP^r$;

\item[$iii)$] $\varphi_H(X)$ is a curve of degree $d$.
\end{itemize}

\noindent
Assume in addition that 
$$
\lambda \ge \frac{\gamma(\gamma +n -2)}{2}.
$$
 Then
\begin{itemize}
\item[$iv)$] 
$m_{d,r} = \gamma -1$ and $ \epsilon_{d,r} = \epsilon$;

\item[$v)$] $\varphi_H(X)$ is an extremal curve in $\mP^r$. 

\end{itemize}
\end{thm}

\begin{proof} Note first that, by
 assumption, $X$ is irreducible and smooth. Then, by Proposition \ref{cuori} and the assumption $\gamma \ge 2$, we have
$$
\lambda > 0, \quad \hbox{if} \quad n =0
$$
and
\begin{equation}
\label{disu}
\lambda \ge \gamma n, \quad \hbox{if} \quad n >0
\end{equation}

\noindent
$i)$
If $n =0$, it is clear that $\beta \ge 1 \iff \lambda -1 \ge \gamma -2$. If this is not the case, we observe that we can change the role of $\gamma$ and $\lambda$, since on $\mF_0 \cong \mP^1 \times \mP^1$ we can choose arbitrarily one of the two rulings.

\noindent
If $n=1$, then $\beta > 1$ if and only if $\lambda -2 > \gamma-2$ and this holds by the assumption that $X$ is not isomorphic to a plane curve.

Indeed, in general, any irreducible curve $\gamma C_0 + \lambda L$ 
on $\mF_1$ satisfies $\lambda \ge \gamma$ by \eqref{disu}. In particular, if $\lambda=\gamma$, we have  $\beta =1$ and $H = C_0 + L$. On one hand, it is clear that the linear system $|H|$ does not induce an embedding of $\mF_1$, as it maps surjectively to $\mP^2$ and corresponds to the contraction of the exceptional curve $C_0$; it is well known that this is  the blowing up morphism
$\mF_1 \to \mP^2$ of the plane in a point. 

On the other hand, under such a morphism, any smooth irreducible curve $X\sim \gamma (C_0 + L)$ is mapped isomorphically to a smooth degree $\gamma$ plane curve, which contradicts our assumption.

\noindent
If $n>1$, obviously $\beta > n$ if and only if
$$
\lambda -n-1 > n(\gamma-2).
$$
i.e. $\lambda > n(\gamma -1) +1$.
But this holds since $n > 1$ implies $n\gamma -n+1 < n\gamma$ and, by (\ref{disu}), $n\gamma \le \lambda$.

 \medskip
\noindent
$ii)$
Again by Proposition \ref{cuori}, the linear system $|H| = | C_0 + \beta L|$ is very ample if and only if $\beta > 0$ if $n =0$ or $\beta > n$ if $n >0$ and this holds by $(i)$. Hence $\varphi_H$ is an embedding.

\noindent
Moreover, we have the well--known formula (see, for instance \cite[Proposition 1.8 - (ii)]{BS}):
\begin{equation}
\label{stellina}
r+1 = h^0( \sO_{\mF_n} (C_0 + \beta L)) = 
2\beta +2 -n.
\end{equation}

 \medskip
\noindent
$iii)$
Therefore $\varphi_H$ embeds $X$ in $\mP^r$ as a curve of degree
$$
(\gamma C_0+ \lambda L) \cdot (C_0 + \beta L)=\gamma(\beta - n) + \lambda = d,
$$
as required.

 \medskip
\noindent
$iv)$
In order to show that 
$m_{d,r} = \gamma -1$ and $ \epsilon_{d,r} = \epsilon$, it is enough to prove that 
$$
d -1 -( \gamma -1)(r-1) = \epsilon \quad \hbox{and} \quad 0 \le \epsilon \le r-2.
$$
Clearly
$$
d -1 -( \gamma -1)(r-1) = \gamma(\beta - n) + \lambda -1 - ( \gamma -1)(2\beta -n ) = \lambda -n-1-\beta(\gamma -2)
$$
and, substituting the value of $\beta$, we obtain the requested equality.

In order to show that $\epsilon \le r-2$, note first that, for any $n \ge 0$, we have
$$
 \epsilon \le r-2 \iff \lambda -n-1-\beta(\gamma -2) \le 2\beta -n -1 \iff
 \lambda \le \beta \gamma.
 $$
Since $\epsilon \le \gamma -3$, clearly
$$
\beta =\frac{\lambda -n-1-\epsilon}{\gamma -2} \ge \frac{\lambda -n +2 - \gamma}{\gamma -2},
$$
so in order to show that $\lambda \le \beta \gamma$, it is enough to prove that
$$
\lambda \le \gamma \frac{\lambda -n +2 - \gamma}{\gamma -2}
\quad \iff \quad
2 \lambda \ge \gamma(\gamma +n -2)
$$
and this holds by assumption.

 \medskip
\noindent
$v)$ 
In order to prove that $\varphi_H(X) \subset \mP^r$ is an extremal curve, we compute the genus $g$ of $X$
using the Adjunction Formula obtaining
\begin{equation}
\label{genrig}
\begin{array}{ll}
2g-2& =(K_{\mF_n} +X) \cdot X =(-2C_0 -(2+n)L + \gamma C_0+\lambda L) \cdot (\gamma C_0+\lambda L)=\\
 &= 2(\lambda \gamma - \lambda -\gamma) - n \gamma (\gamma -1),\\
 \end{array}
\end{equation}
which yields
\begin{equation}
\label{genecaste}
g=\lambda \gamma - \lambda -\gamma +1 - \frac{n}{2} \gamma (\gamma -1) = (\lambda-1)(\gamma -1) - \frac{n}{2} \gamma (\gamma -1). 
\end{equation}
Now we can compute the Castelnuovo bound
$$
\pi(d,r) = m_{d,r}\left( \frac{m_{d,r}-1} 2 (r-1)+\epsilon_{d,r} \right) = 
(\gamma-1)\left( \frac{\gamma -2} 2 (2\beta -n)+\epsilon \right).
$$
Since
$$
2\beta -n = 2 \; \frac{\lambda -n-1-\epsilon}{\gamma -2} - n = \frac{2\lambda -2-2\epsilon - n \gamma }{\gamma -2}
$$
we obtain
$$
\pi(d,r) = (\gamma-1)\left( \frac{1} 2 (2\lambda -2-2\epsilon - n \gamma)+\epsilon \right) = (\gamma-1)\left(\lambda -1-\epsilon -n\gamma/2 + \epsilon
\right) 
$$
and, finally,
$$
\pi(d,r) =(\gamma-1)(\lambda -1 -n\gamma/2 ) = (\gamma-1)(\lambda -1) - \frac{n}{2} \gamma (\gamma -1). 
$$
Comparing this formula with (\ref{genecaste}), we see that $\pi(d,r) = g$ and hence $\varphi_H(X)$ is an extremal curve in $\mP^r$.
\end{proof}

\begin{rmk}
By the irreducibility of $X$, we have
$\lambda \ge \gamma n$ from Proposition \ref{cuori}. As a consequence, the additional assumption in Theorem \ref{everex}
$$
\lambda \ge \frac{\gamma(\gamma +n -2)}{2}.
$$
holds if $n \ge \gamma -2$.
\end{rmk}

%%%%%%%%

In the sequel we will need to relate the gonality of $X$ with its $m$-ratio. Since Theorem \ref{acgh1} claims that the extremal curves, not isomorphic to plane curves, lie on a rational normal ruled surface, we here recall the following result of Martens (see \cite{M}) which describes such a relationship in the wider case of curves on ruled surfaces (see Definition \ref{calcol}).

\begin{thm}
\label{marty} 
Let $X$ be a reduced and irreducible curve on a Hirzebruch surface $\mF_n$ and assume that $X$ is not a fibre. Then the gonality of $X$ is computed by a ruling of the surface, unless $n=1$ and $X \sim \alpha(C_0 + L)$ for some $\alpha \ge 2$, in which case $X$ is isomorphic to a plane curve of degree $\alpha$ and its gonality is $\alpha -1$.
\end{thm}

\begin{rmk}
The exceptional case in Theorem \ref{marty} concerns curves of the type
$$
X \sim \alpha (C_0 + L)
$$
on $\mF_1$. We observe that such a situation never occurs in the framework of extremal curves of type (ii) and (iii) in Theorem 4.1.

Note first that for a rational ruled surface $R \subset \mP^r$ of degree $r-1$ the hyperplane divisor $H$ satisfies 
$H \sim C_0 + \frac{r-1- C_0^2 }{ 2} L$ by Remark \ref{fritto}. Hence
$R \cong \mF_1$ if and only if 
$$
H \sim C_0 + \frac{r}{ 2} L.
$$
It follows that, in case (ii), we have
$$
X \sim mH+L= m C_0 + \left(\frac{mr} { 2} +1\right)L
$$
and, so, $X \sim \alpha C_0 + \alpha L$ if and only if $\alpha=m= \frac{2}{ 2-r}$, which is not possible for $r \ge 2$.

\noindent
In case (iii) we have
$$
X \sim (m+1)H-(r-2-\epsilon)L= (m+1) C_0 +\left(\frac{(m+1)r} { 2}-r+2+\epsilon \right)L.
$$
hence $X \sim \alpha C_0 + \alpha L$ if and only if $\alpha=m+1$ and 
$$
m+1 = \frac{(m+1)r} { 2}-r+2+\epsilon \quad \Rightarrow \quad
\epsilon = (2-r) (m-1)/2.
$$
But $\epsilon \ge 0$ so we get a contradiction for $r \ge 3$.
\end{rmk}

The two results recalled above (Theorems \ref{acgh1} and \ref{marty}) 
lead to the following consequence, whose formulas 
immediately come from (\ref{degm}) and (\ref{maxgen}).

\begin{cor} 
\label{propc}
Let $X$ be a $\gamma$-gonal extremal curve in $\mP^r$ (where $r \ge3$) of degree $d$, genus $g$ and $m$-ratio $m$. If $X$ is not isomorphic to a plane curve, then there exists a rational normal ruled surface $R$ such that $X \subset R \subset \mP^r$ and, setting $\Pic(R) = \mZ[H] \oplus \mZ[L]$, either:

\begin{itemize}

\item [i)] if $X \in |mH+L|$ on $R$ then $\epsilon =0$ and 
$m= \gamma$.
Consequently,
\begin{equation}
\label{deggamma0}
d= \gamma (r-1) +1
\end{equation}
\begin{equation}
\label{maxgengamma0}
g= \frac {\gamma (\gamma-1)(r-1)} 2 .
\end{equation}

or

\item [ii)] if $X \in |(m+1)H- (r-\epsilon-2)L|$ on $R$ then 
$m= \gamma -1$. 
Consequently,
\begin{equation}
\label{deggamma1}
d= (\gamma -1)(r-1)+\epsilon +1
\end{equation}
\begin{equation}
\label{maxgengamma1}
g=(\gamma-1) \left[\frac {\gamma-2} 2 \, (r-1) + \epsilon
\right];
\end{equation}

\end{itemize}

\noindent
In particular, the gonal series $g^1_\gamma$ on $X$ comes from the restriction of the fibration $\pi: R \rightarrow \mP^1$ given by the ruling.
\end{cor}

\begin{rmk} 
\label{bounddeg} 
Assume $\gamma \ge 4$ in Corollary \ref{propc} and consider the case $(ii)$: $m =\gamma -1$.

If $\epsilon = 0$ and $d \ge 3r -1$ then, from (\ref{deggamma1}), we have
$$
d = ( \gamma-1)(r-1) + 1 \ge 3r -1
\quad \Longrightarrow \quad (\gamma-4)r \ge \gamma -3 \quad \Longrightarrow \quad \gamma \ge 5.
$$ 
\end{rmk}

 We shall need the following result (see \cite[Theorem 4.13 ]{LM}).

\begin{thm}
\label{LM1}
Let $X$ be an extremal curve of degree $d \ge 3r - 1$ in $\mP^r$. Then $d_{r -1} = d - 1$ and $X$ does not satisfy all slope inequalities.
\end{thm}

%%%%%%% MAIN
Now we can state the main result of this section, which gives a negative answer to the question posed by Kato and Martens in \cite{KM}, i.e. if the $r$-th slope inequality is violated for extremal curves in $\mP^r$ for any $r \ge 4$.

\begin{thm} \label{mainthm} 
Let $X$ be a $\gamma$-gonal extremal curve of degree $d \ge 3r - 1$ in $\mP^r$ ($r \ge 3$) where $\gamma \ge 4$. Then $d_r=d$.

If, in addition, $X$ is not isomorphic to a plane curve,
then $X$ satisfies the following conditions:
\begin{itemize}
\item[$i)$]
 $d_{r+1}\le d+\gamma -1$;
\item[$ii)$]
by assuming one of the following sets of hypotheses:
\begin{itemize}
 \item[$(a)$] either $\epsilon =0$, $m = \gamma$ and $r \ge \gamma -1$,
 \item[$(b)$] or $\epsilon \ge \gamma -2$, 
\end{itemize}
\noindent
then the $r$-th slope inequality holds, i.e.
$$
\frac {d_r} r \ge \frac {d_{r+1} } {r+1}.
$$ 
\end{itemize}
\end{thm}
%Let $X$ be a $\gamma$-gonal  extremal curve of degree $d \ge  3r - 1$  in $\mP^r$ ($r \ge 3$) where $\gamma \ge 4$. 
% \noindent
% Then $X$  satisfies the following conditions:
% \begin{itemize}
% \item[$i)$]
% $d_r = d$;
% \item[$ii)$]
%  $d_{r+1}\le  d+\gamma -1$.
% \item[$iii)$]
% If, in addition, we suppose that $X$ is not isomorphic to a plane curve and we assume  one of the following sets of hypotheses:
% \begin{itemize}
%  \item[$(a)$] either $\epsilon =0$, $m = \gamma$ and $r \ge \gamma -1$,
%  \item[$(b)$]  or $\epsilon \ge \gamma -2$, 
% \end{itemize}
% \noindent
% then the $r$-th slope inequality holds, i.e.
% $$
% \frac {d_r} r \ge \frac  {d_{r+1} } {r+1}.
% $$ 
% \end{itemize}
% \end{thm}

\begin{proof}
From Theorem \ref{LM1}, we obtain
$d_{r-1}= d-1$. Since $d_r > d_{r-1}= d-1$, we get that $d_r \ge d$. On the other hand $X$ possesses a $g^r_d$ by
assumption, so $d_r \le d$, hence the first statement is proved.

\medskip
\noindent
$(i)$ Since $X \subset \mP^r$ is a curve of degree $d$, denoting by $|H|$ the hyperplane linear system in 
$\mP^r$, the induced linear system $|H_X|$ on $X$ is a linear series $g^r_d$.

From Corollary \ref{propc}, 
 $X$ is a $\gamma$--secant curve on a rational ruled surface $R$ and the gonal series $g^1_\gamma$ on $X$ comes from the restriction of $\pi: R \rightarrow \mP^1$. So, for any $P \in X$, we set 
$\Gamma_P$ to be the gonal divisor contining $P$, i.e. $\Gamma_P = \pi^{-1}(\pi(P)) \cap X$.

Also observe that the general hyperplane $H$ cuts on $R$ an irreducible unisecant curve, say $U_H$. In particular, the general hyperplane $H$ does not contain any line of the ruling. Therefore for a general $H_X \in g^r_d$ and for any $P \in H_X$, we have
$$
H_X \cap \Gamma_P = H \cap X \cap \Gamma_P = (H \cap R)\cap X \cap \Gamma_P = U_H \cap X \cap \Gamma_P= \{P\}.
$$
Setting $\Gamma_P = P + Q_1 + \cdots+ Q_{\gamma-1}$, let us consider the divisor obtained by adding to $H_X$ the $\gamma -1$ further points of the gonal divisor, i.e. $D = H_X + Q_1 + \cdots+ Q_{\gamma-1}$;
we have
$$
\deg D = \deg H_X +\gamma -1 = d+ \gamma -1.
$$
Now let us consider the canonical model $X_K \subset \mP^{g-1}$; here we can apply the Geometric Riemann-Roch Theorem to all the divisors $\Gamma_P$, $H_X$ and $D$. First we obtain
$$
\dim \langle \Gamma_P \rangle = \deg \Gamma_P - h^0({\mathcal O}(\Gamma_P)) = \gamma-2 
$$
and
$$
\dim \langle H_X \rangle = \deg H_X - h^0({\mathcal O}(H_X)) = d-r-1.
$$
Consequently, since the intersection $ \langle \Gamma_P \rangle \cap \langle H_X \rangle $ contains $P$, we have
$$
\dim \langle D \rangle \le \dim \langle H_X \rangle +\gamma -2 = d-r+\gamma -3.
$$
Hence, again from the Geometric Riemann-Roch Theorem, we get that
$$
t+1 := h^0({\mathcal O}(D)) = \deg D - \dim \langle D \rangle \ge d+\gamma -1 -(d-r+\gamma -3) = r+2.
$$
Therefore there exists an integer $t \ge r+1$ such that $|D|$ is of the form $g^t_{d+\gamma -1}$. 
This implies that $d_t \le d+\gamma -1$. From 
Lemma \ref{incre}, we finally obtain that $d_{r+1} \le d_t \le d+\gamma -1$.

\medskip
\noindent
$(ii)$ From the fact that $d_r=d$ and $(i)$, it is enough to show that
$$
\frac {d} r \ge \frac {d+\gamma -1} {r+1},
$$ 
or, equivalently, 
\begin{equation}
\label{tesina}
d \ge r(\gamma -1).
\end{equation}
$(a)$ Assume $\epsilon =0$, $m = \gamma$ and $r +1 \ge \gamma$. Then, using (\ref {deggamma0}), 
$$
d= \gamma (r-1) +1 = r \gamma -\gamma +1 \ge r \gamma -(r+1) +1 = r(\gamma -1)
$$
i.e. (\ref{tesina}), as required.

\noindent
$(b)$
Assume $\epsilon \ge \gamma -2 \ge 1$. Then we express the degree using (\ref {deggamma1}), obtaining
$$
d= (\gamma -1)(r-1)+\epsilon +1 \ge (\gamma -1)(r-1)+\gamma -1 = r(\gamma -1)
$$
so we get (\ref{tesina}), as required.
 \end{proof}
 
 Let us observe that the condition that $X$ is not isomorphic to a plane curve is necessary since, otherwise,  the inequality $(i)$ does not hold (see  
 Remark \ref{remarcob}) and, even assuming $(b)$, the $r$-th slope inequality fails.

\begin{cor}
\label{fasciabianca}

Let $X$ be a $\gamma$-gonal extremal curve  in $\mP^r$, where $\gamma \ge 4$.
If the degree $d$ satisfies
\begin{equation}
\label{tradue}
r(\gamma-1) \le d \le \gamma(r-1)+1 
\end{equation}
then the $r$-th slope inequality holds. 
\end{cor}
\begin{proof}
Let us note first that, under the assumption \eqref{tradue}, the curve $X$ cannot be isomorphic to a plane curve. If so, by Theorem \ref{acgh1}, (i), we would have $r=5$, $d=2k$ and $\gamma =k-1$ by Max Noether Theorem, contradicting \eqref{tradue}.

Now observe that the assumption implies
$$
r(\gamma-1) \le \gamma(r-1)+1
\quad \Rightarrow \quad 
r \ge \gamma -1.
$$
Since $X$ is an extremal curve and $X$ is not isomorphic to a plane curve, by Corollary \ref{propc}, we have two possible cases.

In the first one, $m=\gamma$, $\epsilon =0$ and
$d= \gamma (r-1) +1$ from (\ref{deggamma0}). Since $r \ge \gamma - 1$, by Theorem \ref{mainthm} we obtain that the $r$-th slope inequality holds.

In the second case, $m=\gamma -1$ and $d= (\gamma -1)(r-1)+\epsilon +1$ from (\ref{deggamma1}). The assumption (\ref{tradue}) yields
$$
 (\gamma -1)(r-1)+\epsilon +1 \ge r(\gamma-1)
 \quad \Rightarrow \quad 
 \epsilon \ge \gamma -2.
 $$
Again by Theorem \ref{mainthm}, we obtain that the $r$-th slope inequality holds. 
\end{proof}

The study above and, in particular, Corollary \ref{fasciabianca}, can be summarized in the following table where we shall consider only curves which are not isomorphic to plane curves.

The first column concerns the increasing degree and the last one the $r$-th slope inequality.

In particular, the case $d=2r$, here omitted, has been described in 
Remark \ref{remarcoa}.

The first four lines of the table will be considered in detail in the next section.

  \medskip

\begin{center}
 \begin{tabular}{|c c c c c |} 
 \hline
 $d$ & $\quad \gamma \quad$ & $\quad m \quad$ & $\quad \epsilon \quad$ & \quad $r$-th slope \\ [0.8ex] 
 \hline
 \hline
 $2r+1 \le d \le 3r-3$ \qquad & 3  & 2 & $2 \le \epsilon\le r-2$ & \quad yes (trigonal)  \\ 
 \hline
 $3r-2$  & 3 & 3 & 0 & \quad yes (trigonal)  \\
 \hline
 $3r-2$  & 4 & 3 & 0 & \quad $\star$ \quad   \\
 \hline
 $3r-1$  & 4 & 3 & 1 & \quad no\quad   \\
 \hline
 $3r \le d \le 4r-4$ \qquad &  4  & 3 & $2 \le \epsilon\le r-2$ & \quad yes   \\ 
 \hline
 $4r-3$  & 4 & 4 & 0 & \quad yes   \\
 \hline
 $4r-3$  & 5 & 4 & 0 & \quad   \\
 \hline
 $4r-2$  & 5 & 4 & 1 & \quad   \\
 \hline
$4r-1$  & 5 & 4 & 2 & \quad   \\
 \hline
$4r \le d \le 5r-5$ \qquad &  5  & 4 & $2 \le \epsilon\le r-2$ & \quad yes   \\ 
 \hline
 $5r-4$  & 5 & 5 & 0 & \quad yes   \\
 \hline
$5r-4$  & 6 & 5 & 0 & \quad   \\
 \hline
 $5r-3$  & 6 & 5 & 1 & \quad   \\
 \hline
$5r-2$  & 6 & 5 & 2 & \quad   \\
 \hline
$5r-1$  & 6 & 5 & 3 & \quad   \\
 \hline
 $5r \le d \le 6r-6$ \qquad &  6  & 5 & $2 \le \epsilon\le r-2$ & \quad yes   \\ 
 \hline
 $6r-5$  & 6 & 6 & 0 & \quad yes   \\
 \hline
 ... & ... & ... & ... & ...  \\
 \hline
\end{tabular}
\end{center} 
 \medskip
 \centerline{Table 1}

%%%%%%%%%%%%%%%%%%%%%%%
%%%%%%%%%%%%%%%%% SECTION
%%%%%%%%%%%%%%%%%%%%%%%

 \section {Violating cases}

%In this section we shall assume that $X$ is not isomorphic to a plane curve.
This section concerns the first four lines of the Table 1. More precisely, extremal curves of degree in the initial range are forced to have certain gonality (see Proposition \ref{nostra}). Moreover, the $r$-th slope inequality is violated by extremal curves of degree $d=3r-1$ (see Proposition \ref{altrui}). Finally, we will treat the case of curves of degree $d=3r-2$ in Proposition \ref{casor4} and Theorem \ref{caso3rm2}.

\begin{prop}
\label{nostra}
Let $X \subset \mP^r$ (where $r \ge 3$) be an extremal curve of degree $d$ and genus $g$, not isomorphic to a plane curve and such that 
$$
2r+1 \le d \le 3r -1.
$$
Then:
\item{i)} if  $2r+1 \le  d \le 3r -3$ then  $X$ is   trigonal ($m=2$, $\epsilon \ge 2$);
\item{ii)} if  $d = 3r -2$ then   $X$  is  either trigonal or   fourgonal (where $m=3$, $\epsilon =0$);
\item{iii)} if  $d = 3r -1$ then  $X$  is  fourgonal (where $m=3$, $\epsilon =1$).
\end{prop}

\begin{proof}
Let us recall that $d = m(r-1) + \epsilon +1$ where $m =[(d-1) / (r-1)]$.

\noindent
$(i)$ In this case
$$
 \frac{d-1}{r-1} \le  \frac {3r-4}{r-1} <3
$$
hence $m=2$. So $d = 2(r-1) + \epsilon +1$ and the bound $d \ge 2r +1$ implies $\epsilon \ge 2$. Therefore we are in the case of Corollary \ref{propc} - (ii), hence $X$ admits a $g^1_3$, i.e. it is a trigonal curve.

\noindent
$(ii)$ In this case 
$$
 \frac{d-1}{r-1} =  \frac {3r-3}{r-1} =3
$$
hence $m=3$. Therefore $d = 3(r-1) + \epsilon +1 = 3r-2$ and, so, $\epsilon =0$. By Corollary \ref{propc}  there are two possibilities: 
in case $(i)$ we have $\gamma = m =3$ so the curve $X$ is trigonal. 

\noindent
Otherwise, in case $(ii)$ of Corollary \ref{propc}, we have that $\gamma = m+1 =4$ so $X$ is fourgonal.

\noindent
$(iii)$ In this case 
$$
 \frac{d-1}{r-1} =  \frac {3r-2}{r-1} 
$$
hence $m=3$. Therefore $d = 3(r-1) + \epsilon +1 = 3r-1$ and, so, $\epsilon =1$. So the situation is described by 
Corollary \ref{propc} - (ii) and, in particular, $X$ possesses a $g^1_4$.
\end{proof}

%%%%%%%%%%

Concerning the slope inequalities, the cases described above behave as follows: as far as $X$ is trigonal, all the slope inequalities are fulfilled (see Remark \ref{rem2}). The case $d=3r-1$ is described by Lange--Martens (see \cite[Corollary 4.6]{LM}) as follows.

\begin{prop}
\label{altrui}
For any $r \ge 2$ and any extremal curve $X$ of degree
$d=3r-1$ in $\mP^r$, we have
$$
\frac {d_r} r < \frac {d_{r+1} } {r+1}.
$$ 
\end{prop}

Finally, the case $d=3r-2$ is studied in the next Theorem \ref{caso3rm2}. Its proof will involve the following  two results (see, respectively, \cite[Proposition 4.10 and Lemma 4.8]{LM}).

\begin{prop}
\label{LM3}
Let $g_\delta^s$ be a very ample linear series on $X$ with $\delta \ge 3s - 1$ and $\epsilon_{\delta,s} \neq 0$. If
$g > \pi(\delta, s) - m+ 2$, then $ 2\delta \le g + 3s - 2$.
\end{prop}

\begin{lem}
\label{LM2}
Let $X$ be a curve admitting a $g_\delta^s$ with $\delta \ge 2s-1 \ge 3$ such that
\begin{enumerate}
\item[i)] $d_{s-1} =\delta-1$, and 
\item[ii)] $ 2\delta \le g + 3s - 2$.
\end{enumerate}

\noindent
 Then $d_s = \delta$ and the linear series $g_\delta^s$ is complete and very ample. 
Moreover, if $g_{\delta'}^{s'} = |K_X - g_\delta^s|$ (hence $\delta' = 2g-2-\delta$ and $s' = g-1-\delta+s$), we have 
$s' \ge s$,
$d_{s' +1} \ge \delta'+3$ and, so, 
$$
\frac {d_{s'} }{s'} < \frac {d_{{s'} +1} } {{s'} +1}.
$$ 
 \end{lem}

By Proposition \ref{nostra} extremal curves in $\mP^r$ of degree $d=3r-2$ can be either trigonal or fourgonal. Since we know that trigonal  curves satisfy all slope inequalities, we shall focus the fourgonal case.
We shall see that for $r=4$ extremal curves of degree $3r-2=10$ satisfy all slope inequalities, while for $r \ge 5$, such curves violate the $r$-th slope inequality.

\begin{prop}
\label{casor4}
Let $X \subset \mP^4$ be a fourgonal extremal curve of degree $10$. Then $X$ satisfies all slope inequalities.
\end{prop}

\begin{proof}
Since $X$ is extremal, we have $g(X)= \pi(10,4)=9$. 
The curves of genus $g \le 13$ violating some slope inequality have been classified in \cite[Theorem 3.5 (i)]{KM}.
In particular, in genus $g=9$ the only examples are the extremal curves of degree $8$ in $\mP^3$. The gonality sequence of such curves has been determined in \cite[Example 4.7]{LM}, and in particular it satisfies $d_4=11$. But in our case $X$ possesses a $g^4_{10}$, so $d_4 \le 10$. It follows that no slope inequality is violated.
\end{proof}
Now we turn to the case $r\ge 5$.
\begin{thm}
\label{caso3rm2}
Let $X\subset \mP^r$ with $r \ge 5$ be a fourgonal extremal curve of degree
$
d= 3r-2.
$
Then $d_{r+1} = 3r +1$.

In particular, $X$ violates the $r$-th slope inequality:
$$
\frac{d_r }{ r} < \frac{d_{r+1}}{ {r+1}}.
$$
\end{thm}

\begin{proof}
We note, first, that $X$ is not isomorphic to a plane curve; indeed, in this case we would have $r=5$ and $d=3r-2=13$, while such curves lie on the Veronese surface and hence their degree is even.

Therefore $X$ lies on a rational normal surface.
As observed in Proposition \ref{nostra}, $m_{d,r}=3$, $\epsilon_{d,r} =0$. Therefore by Theorem \ref{acgh1} and Corollary \ref{propc}, the class of $X$ on a ruled surface $R \subset \mP^r$ (of degree $r-1$) is given by 
\begin{equation}
\label{classe di X}
X \sim 4H -(r-2)L
\end{equation}
 and
the genus $g$ of $X$ turns out to be
\begin{equation}\label{pidr}
g= \pi (3r-2,r)= 3(r-1).
\end{equation}

It is not difficult to 
verify that the proof of Theorem \ref{mainthm}-$(i)$ applies also to our case,
and since $\gamma = 4$, we have
$$
d_{r+1} \le d+3 = 3r+1.
$$
To prove that $d_{r+1}  \ge 3r+1$ for $r\ge 5$, we claim that $X$ admits also an embedding in $\mP^{r-2}$ as an extremal curve.
Keeping the notation of Remark \ref{fritto}, if  $H \sim C_0 + \beta L$, then it is straightforward to see that $\beta \ge -C_0^2$ and the equality holds if and only if $R$ is a cone (see, for instance, \cite[Theorem 2.5 and Remark (b)]{R}).

 Since $X$ is irreducible we necessarily have
$$
0\le X \cdot C_0 = (4C_0 +(4\beta -r+2)L)\cdot C_0 = 4C_0^2 + 4\beta -r+2,
$$
so that
$$
\beta \ge -C_0^2+ \frac{r-2 }{ 4}> -C_0^2,
$$
hence 
\begin{equation}\label{beta}
\beta \ge -C_0^2 +1,
\end{equation}
in particular $R$ is not a cone. Therefore we can consider the projection of $R$ from a line $L$ of its ruling. Hence 
consider
the divisor $H':= H -L$ on $R$;
taking into account \eqref{beta} 
we see that the linear system $|H'|$ 
maps $R$ in $\mP^{r-2}$ as a degree $r-3$ rational normal surface or a degree $r-3$ rational normal cone
(see, for instance, \cite[Theorem 2.5 and Remark (b)]{R}). Under such a morphism the image of $X$ has degree
$$
X \cdot H'= (4H -(r-2)L) \cdot (H-L)=4H^2 -(r-2)-4=4(r-1)-r-2=3r-6.
$$
As $m_{3r-6,r-2}=\left [ \frac{3r-7 }{ r-3} \right ] =3$ and $\epsilon_{3r-6,r-2}=2$, the maximal genus is in this case
$$
\pi(3r-6,r-2)=3((r-3)+2)= 3(r-1)=g(X),
$$
which proves the claim.

The above construction provides a divisor $H' \in g^s_\delta:= g^{r-2}_{d-4}$, so we can consider the birational morphism
$$
\varphi_{H'}: \; X \rightarrow Y:= \varphi_{H'}(X) \subset \mathbb P^s
$$
where $Y$ is an extremal curve of genus $g = g(X)= 3r-3$, 
$$
\deg(Y) = \delta = d-4= 3r-6 \quad \hbox{and} \quad s = r-2.
$$
Consequentely, $Y \subset \mathbb P^s$ is smooth and $g^s_\delta$ is very ample. Moreover, since $r \ge 5$ then $s \ge 3$. 

\noindent
Furthermore, since $Y$ is an extremal curve, $\delta \ge 3s -1$ and $\epsilon_{\delta,s} =2$, we can apply both Theorem \ref{LM1} and 
Proposition \ref{LM3}, obtaining, respectively, that 
$d_{s-1} = \delta -1$ and
$2 \delta \le g +3s -2$. 

\noindent
Therefore all the assumptions of 
Lemma \ref{LM2} are verified and from it we obtain that $d_s = \delta$ and the Serre dual series $g_{\delta'}^{s'} = |K_X - g_\delta^s|$ of $g^s_\delta$ satisfies the following relation
$$
d_{s' +1} \ge \delta'+3,
$$ 
where
$$
\delta' = 2g-2-\delta= 2(3r-3)-2 - (3r-6)=3r-2, 
$$
$$
 s' = g-1-\delta+s = 3r-3-1 -(3r-6) + r-2=r.
$$
Therefore the above inequality gives
$$
d_{r +1} \ge 3r+1,
$$ 
as required. 

Finally, since $d_r \le d$, we have
$$
\frac{d_r} {r}\le \frac{d}{r} = \frac {3r-2}{r} < \frac{3r+1}{r+1}= \frac{d_{r+1}}{r+1}.
$$
\end{proof}

%%%%%%%%%%%%%%%%%%%%%%%
%%%%%%%%%%%%%%%%% SECTION
%%%%%%%%%%%%%%%%%%%%%%%

 \section {Fourgonal extremal curves}

The setting of the current section is the following:  $X$ is a fourgonal extremal curve in $\mP^r$ 
(where $r \ge 3$) of degree $d$, genus $g$ and $m$-ratio $m$.

 \begin{rmk} 
 Concerning space curves (i.e. the case $r=3$), let us consider the assumptions (a) and (b) of 
 Theorem \ref{mainthm}-$(ii)$. 
Since $0 \le \epsilon \le r-2=1$, then 
 the only possibility is case (a), where $\epsilon =0$ and $m=\gamma \le r+1 = 4$. For this reason, regarding space curves, Theorem \ref{mainthm} describes only the fourgonal case. More precisely, if $X$ is a fourgonal extremal curve of degree $d$ in $\mP^3$ then it satisfies the following conditions:
\begin{itemize}
\item[(i)]
$d_3 = d$;
\item[(ii)]
 $d_4\le d+3$.
\item[(iii)]
If, in addition, we assume $\epsilon = 0$ and $m=4$, then the $3$rd slope inequality holds, i.e.
$$
\frac {d_3} 3 \ge \frac {d_4 } 4.
$$ 
\end{itemize}
Note that this situation is quite specific. Namely, formulas (\ref{deggamma0}) and (\ref {maxgengamma0}) give $d=9$ and $g=12$. 
So, comparing this fact with (\ref{viol3}) (i.e. the violation of the third slope inequality), which holds for $d \ge 10$, we obtain that this bound on the degree given by Kato--Martens is sharp. Moreover, the above claims fit with the description of the gonal sequence given by the same authors in \cite[Theorem 3.5]{KM}.

\end{rmk}

It is easy to specialize the results of Section 4  to the following results in the fourgonal case.

\begin{cor}
\label{mainthm4} 
Let $X$ be a fourgonal extremal curve of degree $d \ge 3r-1$ in $\mP^r$, where $r \ge 3$. 
Then $X$ satisfies the following conditions:
\begin{itemize}
\item[$i)$]
$d_r = d$;
\item[$ii)$]
 $d_{r+1}\le d+3$.
\item[$iii)$]
If, in addition, we assume that $d \ge 3r$ , then the $r$-th slope inequality holds, i.e.
$$
\frac {d_r} r \ge \frac {d_{r+1} } {r+1}.
$$ 
\end{itemize}
\end{cor}
\begin{proof}
Note first that $X$ cannot be isomorphic to a plane curve (otherwise we would have $r=5$ and $d =2k \ge 14$, so its gonality would be $k-1\ge 6$). Therefore we can apply Theorem \ref{mainthm} and obtain $(i)$ and $(ii)$. Morever, $d \ge 3r$ and $\gamma=4$ imply that $\epsilon \ne 1$ (see Table 1). Hence either assumption $(a)$ or $(b)$ in Theorem \ref{mainthm} hold.
\end{proof}

Moreover, a stronger form of Corollary \ref{fasciabianca} holds.

\begin{cor}
Let $X$ be an extremal curve of degree $d$ in $\mP^r$, where $r \ge 3$. 
 Then the following conditions are equivalent:
 \begin{itemize}
 \item[$i)$] the degree satisfies $3r \le d \le 4r-4$;
\item[$ii)$] $X$ is fourgonal and the $r$-th slope inequality holds.
\end{itemize}
\end{cor}
\begin{proof}
 Immediate, from Table 1.
\end{proof}

As a consequence of the results of the previous sections, we are able to determine the gonal sequence of some extremal curve in a certain interval and to show that all the slope inequalities hold there. Indeed, on one hand, in Theorem \ref{caso3rm2} we have seen that, by projecting an extremal curve lying in $\mP^r$ from a gonal divisor, we may obtain an extremal curve in $\mP^{r-2}$.

On the other hand, for an extremal curve in $\mP^r$, the values of $d_{r}$ and $d_{r-1}$ are determined by Theorems \ref{LM1} and \ref{mainthm}.

\begin{thm}
\label{verylast}
Let 
$$
X \sim 4(C_0 +n L) \subset \mF_n
$$
be a smooth irreducible curve, where $n \ge 3$.
For any integer $a$ such that 
\begin{equation}
\label{questaea}
0\le a \le \left[ \frac{n-3}{ 2}\right], 
\end{equation} 
the gonal sequence of $X$ satisfies
\begin{equation}
\label{seqcp}
d_{n+2a} = 4(n+a)-1, \quad
d_{n+2a+1}= 4(n+a).
\end{equation}
\noindent
Moreover, the following bound holds
\begin{equation}
\label{lastineq}
d_{n+2\left[\frac{n-3}{ 2}\right]+2} \le 4n+4\left[\frac{n-3}{ 2}\right]+3.
\end{equation}
Consequentely, for any $n\le r \le n+2\left[ \frac{n-3}{ 2}\right]+1$, the $r$-th slope inequality holds. 
\end{thm}

\begin{proof}
For any integer $a$ as above, consider the divisor $H_a$ on $\mF_n$ defined by
$$
H_a=C_0 +(n+a)L
$$
and set
$$
r_a:=n+2a+1.
$$
By Proposition \ref{cuori}-(a) and (\ref{stellina}), if $a \neq 0$, then the linear system $|H_a|$ embeds $\mF_n$ isomorphically in ${\mP }^{r_a}$.

\noindent
If $a=0$, then $\varphi_{H_0} (\mF_n)$ is a ruled surface of degree $H_0^2 =n$ in $\mP^{n+1}$. Moreover
$$
\deg(\varphi_{H_0}(C_0)) = C_0 \cdot H_0 = 0
$$
so the unisecant $C_0$ is contracted to a point. Hence $\varphi_{H_0} (\mF_n)$ is a rational normal cone and 
$\varphi_{H_0}$ is an isomorphism between 
$\mF_n \setminus C_0$ and its image.

In both cases, the corresponding morphism $\varphi_{H_a}$ is an isomorphism on $X$. Namely, if $a \neq 0$ it is clear. If $a = 0$, it follows from $X \cdot C_0 = 4(C_0 + nL) \cdot C_0 =0$, so $X \subset (\mF_n \setminus C_0)$.

Therefore, for any $a \ge 0$, the curve $X_a:= \varphi_{H_a}(X) \subset \mP^{r_a}$ has degree
$$
\delta_a :=X \cdot H_a = 4(C_0 +nL)\cdot (C_0 +(n+a)L)=4(n+a).
$$

We claim that $X_a$ is an extremal curve in $\mP^{r_a}$. In order to show this, we compute first the genus of $X$. As usual, the Adjunction formula gives
$$
\begin{array}{ll}
2g-2& =(K_{\mF_n} +X) \cdot X =(-2C_0 -(2+n)L + 4 C_0+4n L) \cdot (4 C_0+4n L)=\\
 &= 12 n -8.\\
 \end{array}
$$
Therefore we obtain 
\begin{equation}
\label{genere di X}
g=6n-3.
\end{equation}
Now we compute the Castelnuovo bound of the genus.
Taking into account the bound (\ref{questaea}) of the integer $a$, we obtain that
$$
m_{\delta_a,r_a}=\left [ \frac {4n+4a-1}{ n+2a}\right]=3, \quad \epsilon_{\delta_a,r_a}=n-2a-1
$$
and
$$
2 \le \epsilon_{\delta_a, r_a} \le n-1.
$$
An immediate computation of the Castelnuovo bound given in (\ref{casbo})
 shows that
$$
\pi (\delta_a, r_a)= 3(n+2a+\epsilon_{\delta_a,r_a})=3(2n-1)=6n-3,
$$
which coincides with the value of $g$ determined in (\ref{genere di X}), so $X_a \subset \mP^{r_a}$ is an extremal curve.

In order to prove (\ref{seqcp}), note first that
 $\epsilon_{\delta_a,r_a} \ge 2$ implies that $n-2a \ge 3$; therefore $\delta_a \ge 3r_a $. So we can apply the cited result of Lange--Martens (see Theorem \ref{LM1}) to the extremal curve $X_a \subset \mP^{r_a}$, obtaining that
$$
d_{n+2a}=\delta_a -1 = 4(n+a)-1.
$$
Second, since $X_a$ is a fourgonal extremal curve and $r_a \ge3$ (again from (\ref{questaea})), we can apply also Corollary \ref{mainthm4}-$(i)$, obtaining that
$$
d_{n+2a+1}=\delta_a = 4(n+a).
$$
Therefore (\ref{seqcp}) is proved. As a consequence, for any $0\le a \le \left[ \frac{n-3}{ 2}\right]$,
both the $(r_a-1)$-th and $r_a$-th slope inequalities are satisfied.

Now consider the highest value $\overline a=\left[\frac{n-3}{ 2}\right]$. Also in this case $X_{\overline a}$ is an extremal fourgonal curve of degree $\delta_{\overline a}$ in 
$\mP^{r_{\overline a}}$. Hence, by
Corollary \ref{mainthm4}-$(ii)$, 
$$
d_{n+2{\overline a}+2} \le 4(n+{\overline a})+3
$$
hence also (\ref{lastineq}) is proved.

Finally, as shown before, $\delta_a \ge 3r_a $ and so Corollary \ref{mainthm4}-$(iii)$ yields the $r$-th slope inequality in the considered range of $r$.
\end{proof}

\begin{rmk}
The gonal subsequence (\ref{seqcp}) and the bound (\ref{lastineq}) 
in Theorem \ref{verylast} can be explicitly written in both the following cases, according to the parity of $n$.

If $n$ is even, i.e. $n = 2k$ for some $k$, then $2\left[ \frac{n-3}{ 2}\right] =2(k-2) = n-4$ so
$$
d_{n}= 4n-1, \ d_{n+1}= 4n, \ d_{n+2}= 4n+3, \ d_{n+3} = 4n +4, \ \dots,
 \ d_{2n-4}= 6n-9, \ d_{2n-3}=6n-8 
$$
and $d_{2n-2}\le 6n-5$.

If $n$ is odd, i.e. $n = 2k+1$ for some $k$, then $2\left[ \frac{n-3}{ 2}\right] =2(k-1) = n-3$ so
$$
d_{n}= 4n-1, \ d_{n+1}= 4n, \ d_{n+2}= 4n+3, \ d_{n+3} = 4n +4, \ \dots,
 \ d_{2n-3}= 6n-7, \ d_{2n-2}=6n-6
$$
and $d_{2n-1}\le 6n-3$.
\end{rmk}

Theorem \ref{verylast} shows that, as far as $n \ge 5$, there are classes of extremal curves in $\mP^n$ which gonal sequence does not follow the pattern given in \cite[Proposition 4.2]{KM}, for extremal curves in $\mP^3$. 

Moreover, using the results in \cite{BS2} and \cite{BeS}, it is possible to verify that the curves of Theorem \ref{verylast} are very special curves in the fourgonal moduli space. Therefore the problem of describing the locus of curves violating the slope inequalities is intriguing and deserves further investigations.

%%%%%%%%%%%%%%%%%%%%%%%%
%%%%%%%%%%%%%%%%%%%%%%%%
%%%%%%%%%%%%%%%%%%%%%%%%
%%%%%%%%%%%%%%%%%%%%%%%%

\end{document}